\documentclass[12pt]{article}
\usepackage{amsmath,amsfonts,amssymb,amsthm, epsfig,
epstopdf, titling, url, array}
\usepackage[utf8]{inputenc}
\usepackage{mathtools}
\usepackage[english]{babel}
\usepackage{hyperref}
\newtheorem{thm}{Theorem}[section]

\newtheorem{cor}{Corollary}[section]
\newtheorem{prop}{Proposition}[section]
\numberwithin{equation}{section}

\def\Z{\mathbb Z}

\def\E{\mathbb E}
 \def\S{\mathbb S}
\def\R{I\!\!R}
\def\H{I\!\!H}

\title{Poisson semigroup and the Gruet formula for the heat kernels on spaces of constant curvature}

\author{Mohamed Vall Ould Moustapha}
\date{
\textsc{Unit\'e de recherche Analyse EDP et Mod\'elisation(A.EDP.M)\\
  D\'epartment de Mathematiques,\\
  Facult\'e des Sciences et Techniques,\\
Universit\'e de Nouakchott,
Nouakchott-Mauritanie.}\\
  \textit{E-mail address}: \texttt{mohamedvall.ouldmoustapha230@gmail.com}.}
\begin{document}
\maketitle
\begin{abstract}
This paper is concerned with the Poisson and heat equations on spaces of constant curvature.
More explicitly we provide new methods for obtaining old and new explicit formulas for
the Poisson and heat semigroups on the Euclidean, spherical  and hyperbolic spaces $\R^n$, $\S^n$ and $\H^n$ . We obtain the Gruet formula for the heat kernels in Euclidean and spherical spaces $\R^n$ and $\S^n$, which are new and we provide a new elementary method to derive the classical
Gruet formula Gruet\cite{Gruet} for the kernel of the heat semigroup on the hyperbolic space $\H^n$.
\end{abstract}
\section{Introduction}
The most important partial differential equations in physics and mathematics are Poisson equation, heat equation and the wave equation.
In the last centenary the solutions of the heat, Poisson and wave equations on spaces of constant curvature
have been studied and computed explicitly and there are many interesting articles  published in the mathematics and physics literature.
The wave equation on spaces of constant curvature is considered by Intissar and Ould Moustapha \cite{Intissar-Ould Moustapha1},\cite{Intissar-Ould Moustapha2}, Bunk et al.\cite{Bunk et al.}, Lax-Phillips \cite{Lax-Phillips}, Ould Moustapha \cite{Ould-Moustapha3}, Abdelhay et al.\cite{A-B-M}. The heat equation on the hyperbolic space is studied by Davies-Mandovalos \cite{Davies-Mandouvalos}, Grigor'yan-Noguchi \cite{Grigoryan-Noguchi}, Ikeda-Matsumoto
\cite{Ikeda-Matsumoto} and Gruet\cite{Gruet}. For the heat equation on spaces of constant curvature see Camporesi\cite{Camporesi}, Taylor\cite{Taylor}, Lohoue-Richener\cite{Lohoue-Rychener} , Matsumoto \cite{Matsumoto} and Anker-Ostellari \cite{Anker-Ostellari}.
For a recent work on Poisson equation on the spherical and hyperbolic space the reader can consult Adam et al. \cite{Adam et al.} and Husam-Ould Moustapha \cite{Husam-Ould Moustapha}, Betancor et al. \cite{Betancor et al.}. For recent work on Gruet formulas (see Nizar \cite{Nizar}).\\
The heat and Poisson kernels are the integral kernels of the heat and Poisson semigroups and thus provides solutions
to the heat and Poisson equation based on the Laplace-Beltrami operators on spaces of constant curvature.
The heat kernel on spaces of constant curvature is a transition
probability density of the Brownian motion on theses spaces.
Let $X_t\in \R^n$ the Brownian motion started
at $x$, then the function
$u(t, x) = Exp[u_0(X_t)]$ solves the Cauchy problem \eqref{Heat-Equation}.
Let $X_y\in \R^n$ be the random variable with Cauchy distribution then the function
$U(y, x) = Exp[U_0(X_y)]$ solves the Cauchy problem \eqref{Poisson-Equation}.
These facts lead to physical
significance and applications.
%%%%%%%%%%%%%%%%%%%%%%%%%%%
This article deals with the Poisson and heat semigroups $e^{t A_n}$ and $e^{-y\sqrt{-A_n}}$ associated to the Laplace Beltrami operators on spaces of constant curvature(Euclidean spherical and hyperbolic spaces), solving the problems:
\begin{equation}\label{Heat-Equation}\begin{cases} A_n u(t, x)=\frac{\partial }{\partial t}u(t, x)=0, (t, x)\in \R^+\times \Omega_n\\  u(0, x) = u_0(x);~~~~u_0\in C^{\infty}_{0}(\Omega_n)
\end{cases}\end{equation}
and
\begin{equation}\label{Poisson-Equation}\begin{cases} A_n U(y, x)+\frac{\partial^{2}}{\partial y^{2}}U(y, x)=0, (y, x)\in \R^+\times\Omega_n\\  U(0, x) = U_0(x);~~~~U_0\in C^{\infty}_{0}(\Omega_n)
\end{cases},\end{equation}
where $\Omega_n=\R^n$, $\S^n$ or $\H^n$, is the Euclidean, spherical and the hyperbolic space respectively and $A_n=\Delta_n^E, \Delta_n^S$ or $\Delta_n^H $  is the corresponding Laplace Beltrami operator.
\begin{prop}\label{1}Let $e^{t A_n}$ and $e^{-y\sqrt{-{A}_n}}$ be the heat and the Poisson semigroup operators on spaces of constant curvature, the following formulas hold:\\
i) $e^{-y \sqrt{-A_n}}=\frac{y}{\sqrt{\pi}}L_u\left[e^{A_n/4 u}u^{-1/2}\right](y^2)$,\\
ii) $e^{t A_n}=(4 t)^{-1/2}L_{y^2}^{-1}\left[\frac{\sqrt{\pi}e^{-y\sqrt{-A_n}}}{y}\right](1/4 t)$,\\
where $L_u(f(u))(v)$ is the classical Laplace transform computed with respect to $u$ with the variable in transform $v$
and $[L_{v}^{-1}f(v)](u)$ is the Laplace inverse transform computed with respect to $v$ with
 variable in inverse transform $u$.\\
\end{prop}
\begin{proof}
To prove i)we use the subordination formula in Strichartz \cite{Strichartz} $p. 50$.\\
\begin{equation}\label{stri1}\frac{e^{-y \lambda}}{y}=\frac{1}{\sqrt{\pi}}\int_0^\infty e^{-u y^2}u^{-1/2}e^{-\lambda^2/4u}du,\end{equation}
or
\begin{equation}\label{stri2}e^{-y \lambda}=\frac{y}{\sqrt{\pi}}L_u
\left[u^{-1/2}e^{-\lambda^2/4u}\right]
(y^2),\end{equation}
taking $\lambda=\sqrt{-A_n}$ in the formula \eqref{stri2}we obtain i).
To see ii) in view of the formula \eqref{stri2} we can write $e^{\frac{-\lambda^2}{4 u}}=u^{1/2}L^{-1}_{y^2}\left(\sqrt{\pi}\frac{e^{-y \lambda}}{y}\right)(u)$,
setting $\lambda=\sqrt{-A_n}$ and $\frac{1}{4u}=t$ in the last formula we have ii).\\
Note that this proposition is contained in the work of Dettman \cite{Dettman} formulas $(1.3)$ and $(1.4)$, see also Brag-Dettman \cite{Brag-Dettman}.
\end{proof}
The remaining of the paper is organized as follows: in section 2, we study the Poisson and heat semigroups, and we give an analogous of the Gruet formulas on the Euclidean space $\R^n$. Section 3 is devoted to the Poisson and heat semigroups on spherical space and the Gruet formula for heat kernel on spherical space  $\S^n$ is obtained. In section 4 we consider the Poisson and heat equations and we give an elementary proof of the classical Gruet formula on the hyperbolic space $\H^n$ .
\section{Heat and Poisson semigroups on the Euclidean space $\R^n$}
In this section we give an old  and new explicit formulas for the heat and Poisson kernels on the Euclidean space $\R^n$.
Let \begin{equation}\label{Heat-Kernel-e}
H^{\E}_n(t, x, x')=(4\pi t)^{-n/2}\exp{\left(-|x-x'|^{2}/4t\right)}
\end{equation}
be the classical heat kernel on the Euclidean space $\R^n$,
we start by the following proposition:
%%%%%%%%%%%%%%%%%%%%%%%%%%%%%%%%%%%%
\begin{prop}\label{2}
Let
 $H^{\E}(t, r)= H^{\E}_n(t, x, x')$, $r=|x-x'|$, be the Euclidean heat kernel, for every $t>0, x, x'\in \R^n$, it holds that:
\begin{itemize}
\item i) $ (-\frac{ \partial}{2\pi  r \partial r})H^{\E}_n(t,  r)=H^{\E}_{n+2}(t, r).$\\
\item ii) $ \int_r^\infty \left(r^2-s^2 \right)^{-1/2} H^{\E}_{n+1}(t,  s)2 s ds=H^{\E}_n(t,  r)$.
\item iii) For $n$  odd, $H^{\E}_n(t,  r)= (-\frac{\partial}{2\pi r\partial r})^{(n-1)/2}\left[\frac{e^{-r^2/4 t}}{(4 \pi t)^{1/2}}\right].$
\item iv)For $ n$   even\\
 $H^{\E}_n
 (t,  r)= \left(-\frac{\partial}{2\pi r \partial r}\right)^{(n-2)/2}\int_r^\infty  \left(s^2- r^2\right)^{-1/2} \frac{e^{-s^2/4 t}}{(4 \pi t)^{3/2}} 2 s ds$.
 \item v) For $ n$   even\\
 $H^{\E}_n(t, r)= \int_r^\infty  \left(s^2- r^2\right)^{-1/2}\left(-\frac{\partial}{2\pi s \partial s}\right)^{(n-2)/2} \frac{e^{-s^2/4 t}}{(4 \pi t)^{3/2}} 2 s ds$.
\end{itemize}
\end{prop}
\begin{proof}
i) $\frac{\partial}{\partial r} H^{\E}_n(t, r)=\frac{-2r}{4 t}\frac{\exp{(-r^{2}/4t)}}{(4\pi t)^{n/2}}= -2r\pi H^{\E}_{n+2}(t, r)$.\\
ii) $ \int_r^\infty \left(s^2-r^2 \right)^{-1/2} H^{\E}_{n+1}(t,  s)s ds=\frac{1}{2}\int_0^\infty z^{-1/2}H^{\E}_{n+1}(t, \sqrt{r^2+z}) dz$\\
 $\frac{1}{2}\int_0^\infty z^{-1/2}H^{\E}_{n+1}(t, \sqrt{r^2+z}) dz=\frac{1}{2}\frac{\exp{(-r^{2}/4t)}}{(4\pi t)^{(n+1)/2}}\int_0^\infty  z^{-1/2} e^{-z/4t}dz=\frac{1}{2} H^{\E}_n(t, r).$\\
The parts iii), iv) and v) are consequences of i) and ii).
\end{proof}
\begin{prop} Let, for $y\in \R^+$ and $ x, x'\in \R^n$, $ P^{\E}_n(y, x, x')$, be the Poisson kernel on the Euclidean space then we have\\
i) $P^{\E}_n(y, x, x')=\frac{y}{\pi^{(n+1)/2}}\int_0^\infty e^{-(|x-x'|^2+y^2)u}u^{(n-1)/2}du.
$\\
ii) $ P^{\E}_n(y, x, x')=\frac{\Gamma((n+1)/2))}{\pi^{(n+1)/2}}\frac{y}{\left(|x-x'|^2+y^2\right)^{(n+1)/2}}.
$
 \end{prop}
 \begin{proof}
 The part i) is a consequence of i) of the Proposition \ref{1} and the part ii) is is a consequence of i).
 \end{proof}

\begin{prop}
 If $P^{\E}_n(y, r)= P^{\E}_n(y, x, x')$, $r= |x-x'|$,
   be the Euclidean  Poisson kernel, then we have
\begin{itemize}
\item{i)} $ (-\frac{ \partial}{2\pi  r \partial r})P^{\E}_n(y,  r)=P^{E}_{n+2}(y, r).$\\
\item {ii)} $ \int_r^\infty \left(r^2-s^2 \right)^{-1/2} P^{\E}_{n+1}(y, s) 2 s ds=P^{\E}_n(y, r)$.
\item{iii)}For $n$  odd, $P^{\E}_n(y, r)= (-\frac{\partial}{2\pi r\partial r})^{(n-1)/2}P^{\E}_1(y, r).$
\item{iv)}For $ n$   even\\
 $P^{\E}_n(y, r)= \left(-\frac{\partial}{2\pi r \partial r}\right)^{(n-2)/2}\int_r^\infty  \left(s^2- r^2\right)^{-1/2} \frac{1}{\pi^2}\frac{y}{\left( s^2+y^2\right)^{2}} 2sds$.
 \item v)For $ n$   even\\
 $P^{\E}_n(y, r)= \int_r^\infty  \left(s^2- r^2\right)^{-1/2} \left(-\frac{\partial}{2\pi s \partial s}\right)^{(n-2)/2}\frac{1}{\pi^2}\frac{y}{\left( s^2+y^2\right)^{2}} 2sds$.
\end{itemize}
\end{prop}
\begin{proof}
This Proposition can be proved using Proposition \ref{1} and Proposition \ref{2}, but we can also give a direct proof:
i)is simple and is left to the reader. For
ii) we have\\
 $\int_r^\infty \left(r^2-s^2 \right)^{-1/2} P^{\E}_{n+1}(y,  s) 2 s ds=
 \int_r^\infty \left(r^2-s^2 \right)^{-1/2}\frac{\Gamma((n+2)/2)}{\pi^{(n+2)/2}}
\frac{y}{(s^2+y^2)^{(n+2)/2}} 2s ds,$\\
by setting $s^2-r^2=z$ in the above integral we can write it as:\\
$\frac{\Gamma((n+2)/2)}{\pi^{(n+2)/2}}y\int_0^\infty z^{-1/2}
(z+r^2+y^2)^{-(n+2)/2} dz$,
using the formula (Magnus et al. \cite{Magnus et al.} p.8):
$B(x, y)=b^x\int_0^\infty t^{x-1}(1+bt)^{-x-y}dt,$
with $x=1/2, y=(n+1)/2$, we obtain the formula ii).
The parts iii), iv) and v) are consequence of i) and ii).
\end{proof}
\begin{thm} For $t\in \R^+$ and $x, x'\in \R^n$, the following formulas hold for the Euclidean heat kernel $H^E_n(t, x, x')$
\begin{align}\label{pre-gruet-e} H^{\E}_n(t, x, x')=\frac{\Gamma((n+1)/2)}{2 i\pi^{n/2+1}(4t)
^{1/2}}\int_{\sigma-i\infty}^{\sigma+i\infty}\frac{e^{\frac{s}{4 t}} }{(|x-x'|^2+s)^{(n+1)/2}} ds
\end{align}
and
\begin{align}\label{gruet-e} H^{\E}_n(t, x, x
')=\frac{\Gamma((n+1)/2)}{\pi^{n/2+1
}(4t)^{1/2}}\int_0^\infty {\cal R} e\left[\frac{e^{(a-i \xi)/4 t}}{(|x-x'|^2+a-i\xi)^{(n+1)/2}}\right]d\xi.
\end{align}
\end{thm}
%%%%%%%%%%%%%%%%%%%%%%%%%%%%%%%%%%%%%%%%%%%%%%%%%%%%%%%%%%%%%%%%%%
\begin{proof} The proof of \eqref{pre-gruet-e} uses essentially ii) of Proposition \ref{1} and the formula ii) of the Proposition 2.2.
Set $y=\sigma+i\xi$ in the formula \eqref{pre-gruet-e} to get, after
splitting the integral into two integrals over $\R^+$ and $\R^-$, the formula \eqref{gruet-e}.
Notice that the formula \eqref{gruet-e} is the Euclidean analogous of the Gruet formula on the hyperbolic space (see Gruet \cite{Gruet}) and we call it
the Euclidean Gruet formula.
\end{proof}
The integrand on the right hand side of \eqref{pre-gruet-e}is
a meromorphic function in s when n is odd and we can apply the residue
calculus.
For $n=2k+1$,   $\forall t>0, x, x'\in \R^n$, $r=|x-x'|$, we obtain
$$H_{2k+1}^{\E}(t, r)=\left(\frac{\partial}{-2\pi r\partial r}\right)^k \frac{e^{-r^2/4t}}{\sqrt{4\pi t}}.$$

%%%%%%%%%%%%%%%%%%%%%%%%%%%%%%%%%%%%%%%%%%%%%%%%%%%%%%%%%%%%%%%%%%
\section{Heat and Poisson semigroups on the spherical space}
We recall some facts about the heat kernel on spherical spaces see (Nowak et al.\cite{Nowak et al.} and Camporesi \cite{Camporesi}):
 For $t\in \R^+$ and $\omega, \omega'\in \S^n$,
let $H^{S}_n(t, \varphi(\omega, \omega'))= H_n^{S}(t, \omega, \omega')$  be the  spherical  heat kernel, then we have
\begin{itemize}
\item i) $ (-\frac{\partial}{2\pi\sin \varphi \partial \varphi})H^{S}_n(t, \varphi)=H^{\R}_{n+2}(t, \varphi).$

\item{ii)}For $n$  odd, $H^{S}_n(t, \varphi)= (-\frac{\partial}{2\pi\sin \varphi\partial \varphi})^{(n-1)/2}H_1^{S}(t, \varphi),$
where \begin{equation}H_1^{S}(t, \varphi)=(4\pi t)^{-1/2}\sum_{n\in \Z}e^{-(\varphi+2n\pi)^2/4t}\end{equation}
\item{iii)}For $n$ even, $H^{S}_n(t, \varphi)= (-\frac{\partial}{2\pi\sin \varphi\partial \varphi})^{(n-2)/2}H_2^{S}(t, \varphi).$
where \begin{equation}H_2^{S}(t, \varphi)=(4\pi t)^{-3/2}\sum_{n\in \Z}(-1)^n \int_{\varphi}^\pi \frac{e^{-(\psi+2n\pi)^2/4t}(\psi+2\pi n)}{(\sin^2 \psi/2-\sin^2 \varphi/2)^{1/2}}d\psi.
\end{equation}
%\item ii)$H^{S}_n(t, \varphi)= c_n\int_{-1}^{1}K^{S}_{2n-1}(t/4, \left(\arccos(v \cos\varphi/2)\right)) (1-v^2)^{(n-3)/2}dv$.\\
   % with $c_n=\frac{\pi^{(n-1)/2}}{2^{n-1}\Gamma(n-1)/2)}$
%\item iii) $H^{S}_n(t, \varphi)=
 %\frac{c_n}{2}\int_{\varphi/2}^{\pi-\varphi/2} \left(\cos\varphi-\cos 2\psi\right)^{(n-3)/2}H^{S}_{2n-1}(t/4, \psi)\ sin(\psi) d\psi $.\\
    %with $\tilde{c}_n=\frac{c_n}{2^{n-3}}$
.\end{itemize}
Recall that the Poisson kernel on the sphere $S^n$ is given by Adam et al.\cite{Adam et al.} see also Taylor \cite{Taylor} formula(4.9) p.114.
\begin{equation} P_n^{S}\left(y, \omega, \omega'\right)=\frac{\Gamma((n+1)/2)}{\pi^{(n+1)/2}}\frac{\sinh y}
 {\left(2\cosh y-2\cos d(\omega, \omega')\right)^{(n+1)/2}}.\end{equation}
\begin{prop}  For $t\in \R^+$ and $\omega, \omega'\in \S^n$,
let $P^{S}_n(y, \varphi(\omega, \omega'))= P_n^{S}(y, \omega, \omega')$  be the  spherical  Poisson kernel, then we have
\begin{itemize}
\item i) $ (-\frac{\partial}{2\pi\sin \varphi \partial \varphi})P^{S}_n(y, \varphi)=P^{S
}_{n+2}(y, \varphi).$
\item ii)$P^{S}_n(y, \varphi)= c_n\int_{-1}^{1}\cosh(y/2)
P^{S}_{2n+1}(y/2, \left(\arccos(v \cos
\varphi/2)\right))(1-v^2)^{(n-1)/2}dv$.
\item iii) $P^{S}_n(y, \varphi)=
 \frac{c_n}{2}\int_{\varphi}^{2\pi-\varphi}\frac{\cosh y/2}{\cos\varphi/2}
 \left(1-\frac{\cos^2 \varphi/2}{\cos^2\psi/2}\right)^{(n-1)/2}P^{S}_{2n+1}(y/2, \psi/2)\ sin(\psi/2) d\psi
    $,\\
    with $c_n=\frac{\pi^{(n+1)/2}}{2^{n-1}\Gamma((n+1)/2)}$
\item iv)For $n$  odd, $P^{S}_n(y, \varphi)= (-\frac{\partial}{2\pi\sin \varphi\partial \varphi})^{(n-1)/2}P_1^{S}(y, \varphi).$
\item v)For $n$ even, $P^{S}_n(y, \varphi)= (-\frac{\partial}{2\pi\sin \varphi\partial \varphi})^{(n-2)/2}P_2^{S}(y, \varphi).$
\end{itemize}
\end{prop}
\begin{proof} i) is simple, to prove the part ii), set\\
$I=c_n\int_{-1}^{1} \cosh y/2P^{S}_{2n+1}(y/2, \left(\arccos(v \cos
\varphi/2)\right)) (1-v^2)^{(n-1)/2}dv.\\
I=c_n\frac{\Gamma(n+1)}{(2\pi)^{n+1}}\sinh(y/2)\int_{-1}^{1}(\cosh y/2-v\cos \varphi/2)^{-(n+1)}(1-v^2)^{(n-1)/2}dv.$\\
 Setting $v=1-2\xi$, in the last integral we have\\
 $I=c_n\frac{\Gamma(n+1)}{2\pi^{n+1}}\sinh (y/2)(\cosh y/2-\cos \varphi/2)^{-(n+1)}\times\\
 \int_0^1(1+2\xi\frac{\cos\varphi/2}{\cosh y/2-\cos \varphi/2})^{-(n+1)}[\xi(1-\xi)]^{(n-1)/2} d\xi.$\\
 Using the formula (Magnus et al. \cite{Magnus et al.}, p. 54).
\begin{equation}
 {}_2F_1\left(\alpha, \beta, \gamma, z\right)=\frac{\Gamma(\gamma)}{\Gamma(\beta)\Gamma(\gamma-\beta)} \int_0^1 u^{\beta-1} (1-u)^{\gamma-\beta-1} \left(1-u z\right)^{-\alpha}
du, \end{equation}
with
$Re \gamma> Re \beta >0$ and $\arg(1-z)<\pi$, we obtain\\
 $I=c_n\frac{[\Gamma(n+1)]^2}{2\pi^{n+1}}\sinh (y/2)(\cosh y/2-\cos \varphi/2)^{-(n+1)}\times\\ F(n+1,(n+1)/2, n+1, \frac{-2\cos\varphi/2}{\cosh y/2-\cos \varphi/2}).$\\
 Now
  using the formula (Magnus et al. \cite{Magnus et al.}, p. 38).  ${}_2F_1\left(\alpha, \beta, \alpha, z\right)=(1-z)^{-\beta}$ the proof of ii) is finished.
\end{proof}
\begin{thm}\label{Heat-Spheric} For every $n\geq 2, t>0, \omega, \omega'\in \S^n$, it holds that:

 \begin{equation}\label{pre-gruet-s}H^S_n(t, w, w')=c_n(4t)^{-1/2} \int_{\sigma-i\infty}^{\sigma+i\infty}\frac{\exp{\left(\frac{y^2}{4t}\right)}\, \sinh y}{(\cosh y-\cos \varphi)^{(n+1)/2}}dy,\end{equation}
\begin{align} \label{gruet-s}H_n^S(t, \varphi)=2c_n(4t)^{-1/2}\int_0^{+\infty}{\cal R}e\left[\frac{e^{(\sigma-i\xi)^2/4 t} \sinh(\sigma-i\xi)}{(\cosh (\sigma-i\xi)-\cos\varphi)^{(n+1)/2}}\right] d\xi,
\end{align}
where $c_n=\frac{\Gamma((n+1)/2)}{2^{(n+3)/2}i\pi^{(n/2+1)}}$.
\end{thm}
 \begin{proof} The proof of \eqref{pre-gruet-s} uses essentially ii) of Proposition \ref{1}. To see \eqref{gruet-s}
Set $y=\sigma+i\xi$ in the formula \eqref{pre-gruet-s} and
split the integral into two integrals over $\R^+$ and $\R^-$.
Notice that the formula \eqref{gruet-s} is the spherical analogous of the Gruet formula on the hyperbolic space Gruet \cite{Gruet} and we call it
the spherical Gruet formula.
\end{proof}
The integrand on the right hand side of \eqref{pre-gruet-s}is
a meromorphic function in $s$ when $n$ is odd and we can apply the residue
calculus to obtain,
for n=2k+1,
$$H_{2k+1}^{\S}(t, x, x')=
\left(\frac{\partial}{-2\pi \sin r\partial r}\right)^k \sum_{n\in \Z}\frac{e^{-(\varphi+2\pi n)^2/4t}}{\sqrt{4\pi t}}.$$

%%%%%%%%%%%%%%%%%%%%%%%%%%%%%%%%%%%%%%%%%%%%%%%%%%%%%%%%%%%%%%%%%%%%%

\section{Heat and Poisson semigroups on the hyperbolic spaces $\H^n$}

\begin{prop} (Grigoryan-Noguchi \cite{Grigoryan-Noguchi}, Davies-Mandouvalos \cite{Davies-Mandouvalos}) For every $n\neq 2, t>0, w, w'\in \H^n$, it holds that:
Let $H^{H}_n(t, \rho)= H_n^{H}(t, w, w')$  be the real hyperbolic  heat kernel, then we have
\begin{itemize}
\item i) $ (-\frac{\partial}{2\pi\sinh \rho \partial \rho})H^{H}_n(t,  \rho)=H^{H}_{n+2}(t, \rho).$
\item ii) $ \int_{\rho} ^\infty \left(\cosh^2s/2-\cosh^2 \rho/2\right)^{-1/2}H^{H}_{n+1}(t, s)\ sinh s ds= H^{H}_n(t, \rho)$.
\item iii) For $n$  odd, $H^{H}_n(t, \rho)= (-\frac{\partial}{2\pi\sinh \rho\partial \rho})^{(n-1)/2}\left[\frac{e^{-\rho^2/4 t}}{(4 \pi t)^{1 /2}}\right].$
\item iv)For $ n$   even\\
 $H^H_n(t, \rho)= \left(-\frac{\partial}{2\pi\sinh \rho \partial \rho}\right)^{(n-2)/2}\int_r^\infty  \left(\cosh^2s/2-\cosh^2 \rho/2\right)^{-1/2} \frac{e^{-s^2/4 t}}{(4 \pi t)^{3/2}} sds$.
 \item v)For $ n$   even\\
 $H^H_n(t, \rho)= \int_{\rho}^\infty\left(\cosh^2s/2-\cosh^2 \rho/2\right)^{-1/2} \left(-\frac{\partial}{2\pi\sinh s \partial s}\right)^{(n-2)/2}\frac{e^{-s^2/4 t}}{(4 \pi t)^{3/2}} sds$.
\end{itemize}
\end{prop}
\begin{prop} Let  $P_n^H(y, w, w')$ the Poisson kernel on the hyperbolic space given in Taylor \cite{Taylor} and Adam et al. \cite{Adam et al.}
\begin{align} P_n^H(y, w, w')=\frac{\Gamma((n+1)/2)}{(2\pi)^{(n+1)/2}}\frac{\sin y}{\left(\cosh \rho(w, w')-\cos y\right)^{(n+1)/2}}
  \end{align} and
set $P^H_n(y, \rho)= P^H_n(y, w, w')$  $\rho=d(w, w')$, then we have
Let $H^{H}_n(t, \rho)= H_n^{H}(t, w, w')$  be the real hyperbolic  heat kernel, then we have
\begin{itemize}
\item i) $ (-\frac{\partial}{2\pi\sinh \rho \partial \rho})P^{H}_n(t,  \rho)=P^{H}_{n+2}(t, \rho).$
\item ii) $ \int_{\rho} ^\infty \left(\cosh^2s/2-\cosh^2 \rho/2\right)^{-1/2}P^{H}_{n+1}(t, s)\ sinh s ds= P^{H}_n(t, \rho)$.
\end{itemize}
\end{prop}
\begin{proof}
The proof of this Proposition is simple and is left to the reader.
\end{proof}
%%%%%%%%%%%%%%%%%%%%%%%%%%%%%%%%%%%%%%%%%%%%%%%%%%%%%%%%%%%%%%%%%%%
%%%%%%%%%%%%%%%%%%%%%%%%%%%%%%%%%%%%%%%%%%%%%%%%%%%%%%%%%%%%%%%%%%%%%%%%%%%%%%%%%%%%%%%%%

\begin{thm}\label{Heat-Hyperbolic} The heat kernel on the hyperbolic space is given by \\
 \begin{equation}\label{pre-gruet-h}H^H_n(t, w, w')=C_n (4t)^{-1/2}\int_{\sigma-i\infty}^{\sigma+i\infty}\frac{\exp{\left(\frac{y^2}{4t}\right)}\, \sin y}{(\cosh\rho(w, w')-\cos y)^{(n+1)/2}}dy,\end{equation}
 with $C_n=\frac{\Gamma((n+1)/2)}{2^{(n+3)/2}i\pi^{(n/2+1)}}.$
\end{thm}
\begin{proof} The proof of \eqref{pre-gruet-h} uses essentially ii) of Proposition \ref{1}.
\end{proof}
\begin{cor} (Gruet formula)
\begin{align} \label{gruet-h1}H^H_n(t, w, w')=\frac{K_n}{(4t)^{1/2}}\int_0^{+\infty}{\cal R}e\left[\frac{e^{(\sigma-i\xi)^2/4 t} \sin(\sigma-i\xi)}{(\cosh \rho-\cos(\sigma-i\xi))^{(n+1)/2}}\right] d\xi,
\end{align}
$K_n=\frac{\Gamma((n+1)/2)}{2^{(n+3)/2}\pi^{(n/2+1)}}$.
\begin{align} \label{gruet-h2}H_n(t, z, z')=K_n t^{-1/2}\int_0^{+\infty}\frac{e^{(\pi^2-\xi^2)/2 t} \sinh\xi \sin \pi \xi/t}{(\cosh \rho +\cosh \xi
)^{(n+1)/2}} d\xi,
\end{align}
$K_n=\frac{\Gamma((n+1)/2)}{2^{n/2}\pi^{n/2+1}}$.
\end{cor}
\begin{proof}
To see \eqref{gruet-h1},
set $y=\sigma+i\xi$ in the formula \eqref{pre-gruet-h}, we get after
splitting the integral into two integrals over $\R^+$ and $\R^-$
we obtain the result.
Set $\sigma=\pi$ and replace $t$  by $t/2$  in \eqref{gruet-h1} we obtain the \eqref{gruet-h2}.
Notice that the formula \eqref{gruet-h2} is the Gruet formula on the hyperbolic space(Gruet \cite{Gruet}).
While the classical expressions for $H^H_n(t, r)$ have different forms for odd and even
dimensions, but the Gruet's formula below holds for every $n$.
\end{proof}
The integrand on the right hand side of\eqref{pre-gruet-h} is
a meromorphic function in $s$ when $n$ is odd and we can apply the residue
theorem to obtain,
for$ n=2k+1 $,
$$H_{2k+1}^{\H}(t, x, x')=\left(\frac{\partial}{-2\pi \sinh r\partial r}\right)^k \frac{e^{-r^2/4t}}{\sqrt{4\pi t}}.$$

%%%%%%%%%%%%%%%%%%%%%%%%%%%%%%%%%%%%%%%%%%%%%%%%%%%%%%%%%%%%%%%%

\end{document}